\documentclass[12pt]{article}
\usepackage{amsmath,amssymb,latexsym}
\usepackage{enumerate}
\usepackage{amsthm}
\usepackage{graphicx}
\usepackage{epsfig}

\newcommand\NN{\mathrm{I\!N}}
\newcommand\RR{\mathrm{I\!R}}
\newcommand\CC{\mathbb{C}}

\newtheorem{theorem}{Theorem}[section]
\newtheorem{definition}[theorem]{Definition}
\newtheorem{corollary}[theorem]{Corollary}

\newtheorem{remark}[theorem]{Remark}
\newtheorem{proposition}[theorem]{Proposition}

\newcommand{\address}{Address: Department of Mathematics, University of North Texas, 1155 Union Circle \#311430, Denton, TX 76203-5017, USA; E-mail: tobeymathis@my.unt.edu, kiko.kawamura@unt.edu}

\numberwithin{equation}{section}

\title{Revolving sequences and Terdragon}

\author{Tobey Mathis and Kiko Kawamura \\University of North Texas \footnote{\address}}

\begin{document}

\maketitle

\begin{abstract}
In 1970, Davis and Knuth introduced the concept of revolving sequences to represent Gaussian integers. Much later, Kawamura and Allen recently generalized this idea to a wider class of revolving sequences that parametrize certain self-similar fractals including the Levy Dragon and Tiling Dragon, which are the unique compact solution of certain families of Iterated Function Systems.

In this paper, we build on the work of Kawamura and Allen to include a wider collection of Iterated Function Systems and introduce a new type of revolving sequence which parametrizes a different family of self-similar fractals including the Terdragon.
\end{abstract}

\section*{Introduction}

The history of systematic mathematical research on self-similar fractals dates back to 1981, when Hutchinson proved the following celebrated theorem~\cite{Hutchinson-1981}: For any finite family of similar contractions $\{\psi_{1},\psi_{2},\dots,\psi_{m}\}$ on $\RR^{n}$, which is called an iterated function system or IFS, there exists a unique non-empty compact solution $X$ of the set equation:
$$X=\bigcup_{i=1}^{m}\psi_{i}(X).$$ 
We call $X$ an attractor or a self-similar set for the IFS. 

\medskip
Before Hutchinson, it is worth mentioning that Williams~\cite{Williams-1971} in 1971 essentially proved that the following set is a self-similar set.  
\begin{equation*}
X=\text{closure} \Big( \bigcup_{1 \leq i_1, i_2, \cdots i_n \leq m}\text{Fix} (\psi_{i_1} \circ \psi_{i_2} \circ \cdots \circ \psi_{i_n}) \Big),
\end{equation*}
where $\text{Fix}(\psi_i)$ is the unique fixed point of $\psi_i$. 
\medskip

Around the same time, in 1970, C.~Davis and D.~E.~Knuth~\cite{Davis+Knuth-1970} introduced the notion of {\it revolving representations} of a Gaussian integer: for any $z=x+iy$ with $x, y \in \mathbb{Z}$,  there exists a revolving sequence $(\delta_{0}, \delta_{1},\dots \delta_{n})$ such that  
\begin{equation*}
  z=\sum_{k=0}^{n} \delta_{n-k}(1+i)^{k}, 
\end{equation*}
where $\delta_{k} \in \{0, 1, -1, i, -i\}$ with the restriction that the non-zero values must follow the cyclic pattern from left to right:
$$\cdots \to 1 \to (-i) \to (-1) \to i \to 1 \to \cdots.$$
Notice that a revolving sequence is created by a 90 degree angle of rotation.

\medskip
Recently, Kawamura and Allen~\cite{Kawamura-Allen-2020} defined {\it generalized revolving sequences}, where the 90 degree angle of rotation is replaced with a more general angle $\theta$ where $|\theta|=\frac{2\pi q}{p}\leq\pi$. In other words, $\delta_{k} \in \{0, 1, e^{i \theta}, e^{2i \theta}, \cdots, e^{(p-1)i \theta}\}$ with the restriction that the non-zero values must follow the cyclic pattern from left to right:
$$\cdots \to 1 \to e^{i \theta} \to e^{2i \theta} \to e^{3i \theta} \to e^{4i \theta} \to \cdots.$$ 

Define $W_{\theta}$ as the set of all generalized revolving sequences with parameter $\theta$. Kawamura and Allen considered a relationship between $W_{\theta}$ and self-similar sets and found new parametrized expressions for certain self-similar sets. 

\begin{theorem}[Kawamura-Allen]
Let $X_{1,\alpha,\theta}$ be the self-similar set generated by the iterated function system (IFS): 
\begin{equation}
\label{eq:IFS1} 
\begin{cases}
 \psi_1(z)=\alpha z,& \\
 \psi_2(z)=(\alpha e^{i \theta}) z+\alpha,
\end{cases}
\end{equation}
where $\alpha \in \CC$ such that $|\alpha|<1$.  

The self-similar set $X_{1,\alpha,\theta}$ has the following parametrized expression: 
\begin{equation}
X_{1,\alpha,\theta}=\left\{ \sum_{n=1}^{\infty} \delta_{n}\alpha^{n}: \delta_{j_1}=1, (\delta_{1}, \delta_{2},\dots )\in W_{\theta} \right\},
\label{eq:X1}
\end{equation}
where $j_{1}:=\min\{j: \delta_{j} \not= 0\}$.
\end{theorem}
Notice that each generalized revolving sequence  determines a complex number and $X_{1,\alpha,\theta}$ is a set of points in the complex plane. 
In particular, if $\alpha=(1-i)/2$ and $\theta=-\pi/2$, $X_{1,\alpha,\theta}$ is a Tiling dragon~\cite{Ito-1987}. If $\alpha=(1-i)/2$ and $\theta=\pi/2$, then $X_{1,\alpha,\theta}$ is L\'evy's dragon curve, which is a continuous curve but with positive area~\cite{Kawamura-2002}.

\begin{figure}
\includegraphics[height=4cm,width=4cm, bb=0 0 1000 1000]{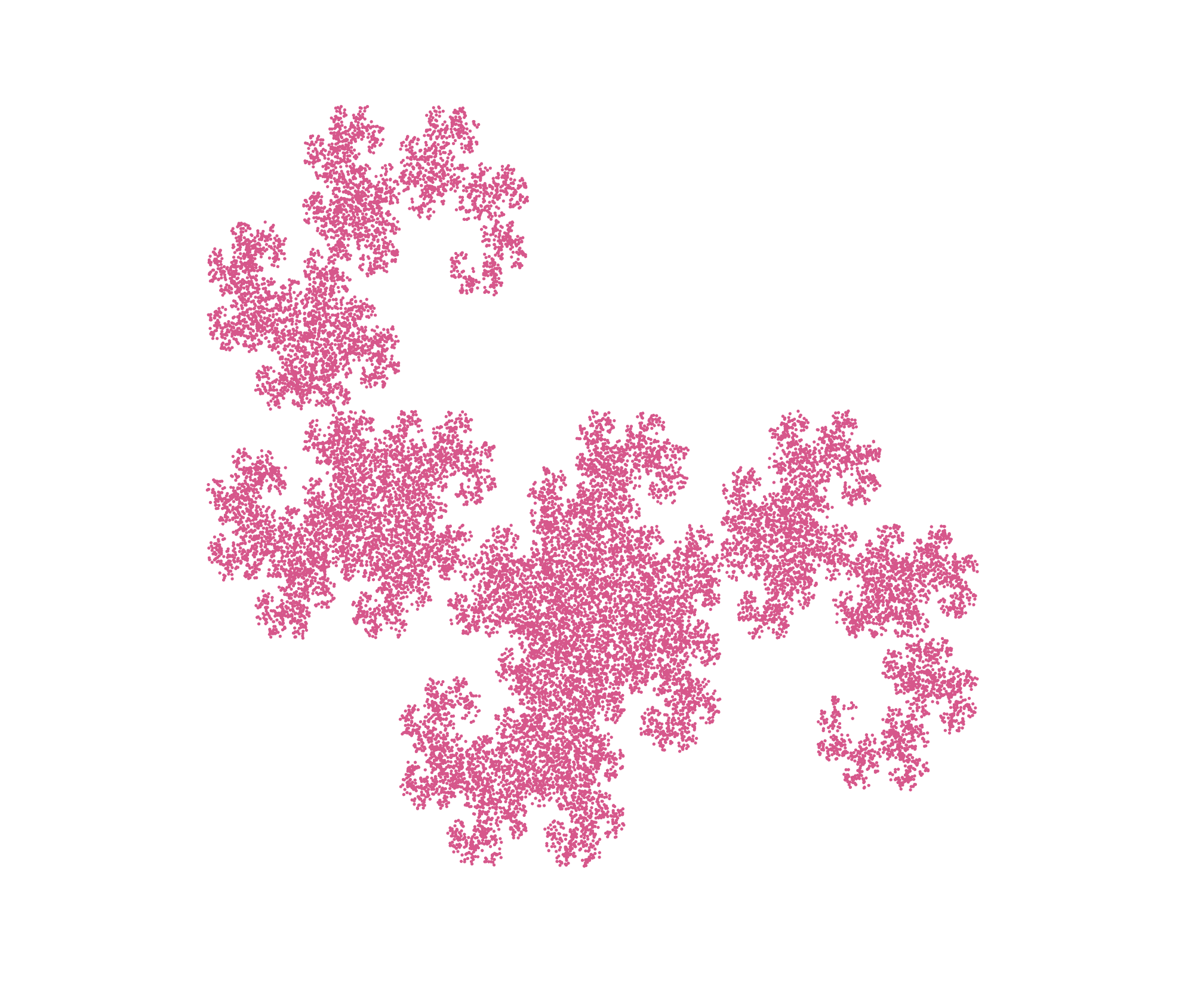}
\qquad \qquad \qquad
\includegraphics[height=4cm,width=4cm, bb=0 0 1000 1000]{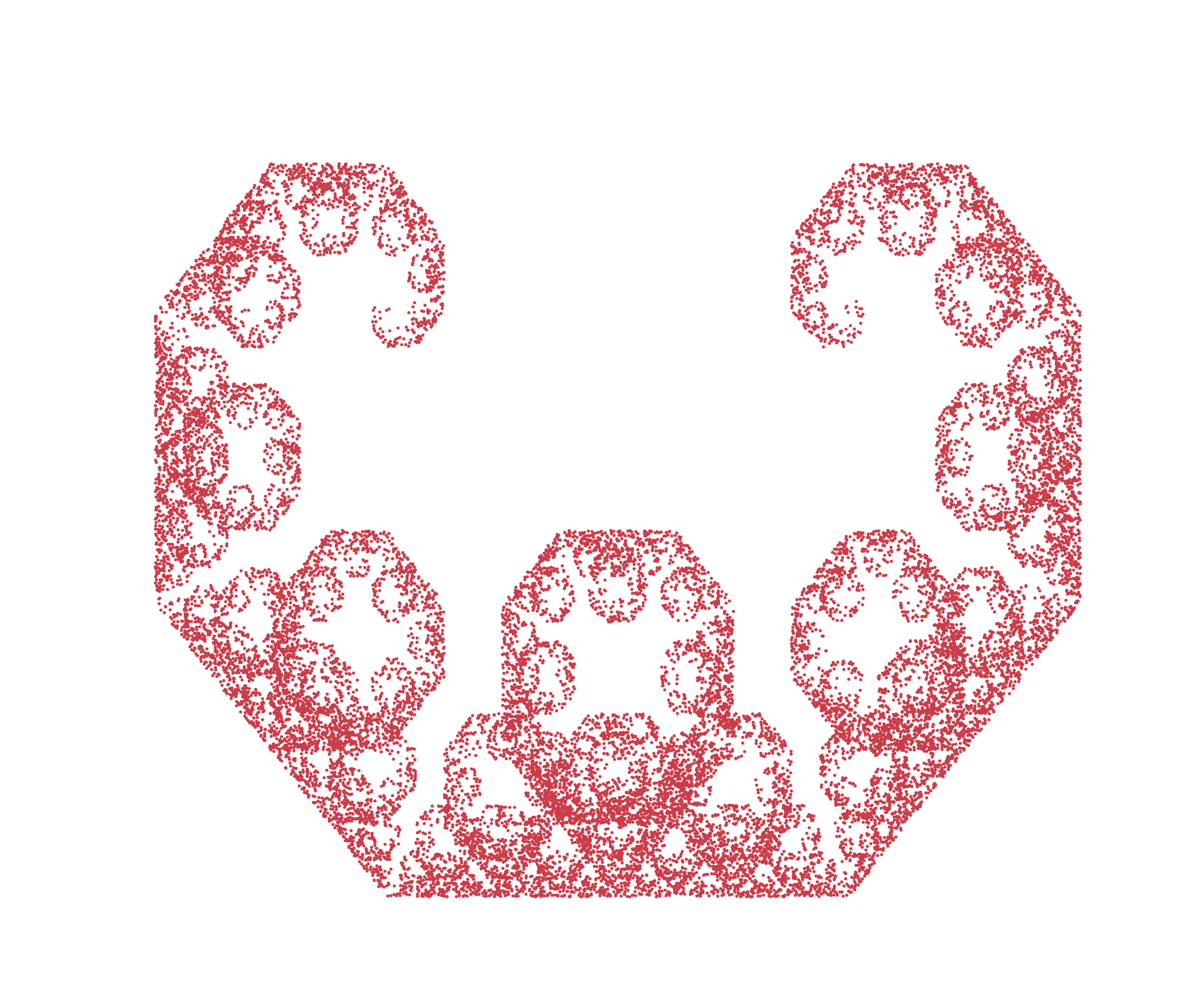}
\caption{$X_{1,\alpha, \beta, \theta}:\alpha=\frac12-\frac{i}{2}, \theta=-\pi/2$ (left) and $\theta=\pi/2 $ (right).} 
\label{figure0} 
\end{figure}

\medskip
In addition to these two dragons, there are three other famous dragons: the Heighway dragon, the Twindragon and Terdragon. (See \cite{Edgar-1990} for more details). Motivated by these three dragons, in this paper we build on the work of Kawamura and Allen to include a wider collection of self-similar sets. 

Section 1 gives the parametric expression of certain self-similar sets, including the Heighway dragon and the Twindragon, generated by the IFS: 
\begin{equation*}
\begin{cases}
 \phi_1(z)=\alpha z,& \\
 \phi_2(z)=(\alpha e^{i \theta}) z+1.
\end{cases}
\end{equation*}
Interestingly, the parametric expression is given by only changing the initial index in the sum from $n=1$ to $n=0$ in \eqref{eq:X1}. This result suggests a natural question: Does changing the constant term in the IFS from $\alpha$ to $1$ always generate a geometrically similar attractor? Is it always true that this small change does not influence the parametrized expression for certain self-similar sets? 
These questions are discussed in Section 2. 

In section 3, we finally study the Terdragon, a more challenging self-similar set generated by the following set of three contractions:
\begin{equation*}
\begin{cases}
 \phi_1(z)=\alpha z,& \\
 \phi_2(z)=(\alpha e^{i \theta}) z+\alpha, & \\
 \phi_3(z)=\alpha z + \beta,
\end{cases}
\end{equation*}
where $\alpha=1/2-(\sqrt{3}/6)i, \beta=1/2+(\sqrt{3}/6)i$ and $\theta=2\pi/3$. We introduce a new type of revolving sequence which parametrizes this different family of self-similar sets. 

\section{The Heighway dragon and Twindragon}

Before stating our results, some notation needs to be introduced. Let $\alpha \in \CC$ denote a complex parameter satisfying $|\alpha|<1$. Let $\theta$ be an angle with $-\pi < \theta \leq \pi$ and a rational multiple of $2 \pi$; that is $|\theta|=(2 \pi q)/p$ where $p \in \NN, q \in \NN_0$. Define 
$$\Delta_{\theta}:= \{0, 1, e^{i \theta}, e^{2i \theta}, \cdots e^{(p-1)i \theta}\}.$$

\begin{definition}
A sequence $(\delta_{1},\delta_{2},\dots)\in \{0, 1, e^{i \theta}, e^{2i \theta}, \cdots, e^{(p-1)i \theta}\}^{\NN}$ satisfies the 
{\it Generalized Revolving Condition (GRC)} if the subsequence obtained after the removal of its zero elements is a (finite or infinite) truncation of the sequence $(e^{n i \theta})$. More precisely, let $(n_i)$ be the sequence of indices $n$ such that $\delta_n \neq 0$, listed in increasing order. Then $\delta_{n_{i+1}}=e^{i \theta} \delta_{n_i}$ for each $i$. 
\end{definition}

The Heighway dragon was first discovered by a physicist J.~Heighway. M.~Gardner described that the Heighway dragon is the shape that is generated from repeatedly folding a strip of paper in half~\cite{Gardner-1967}. Many of its properties were first investigated by Davis and Knuth. Motivated by the Heighway dragon, Davis and Knuth also constructed another dragon, called the Twindragon~\cite{Davis+Knuth-1970}.
\begin{figure}
\includegraphics[height=4cm,width=4cm, bb=0 0 1000 1000]{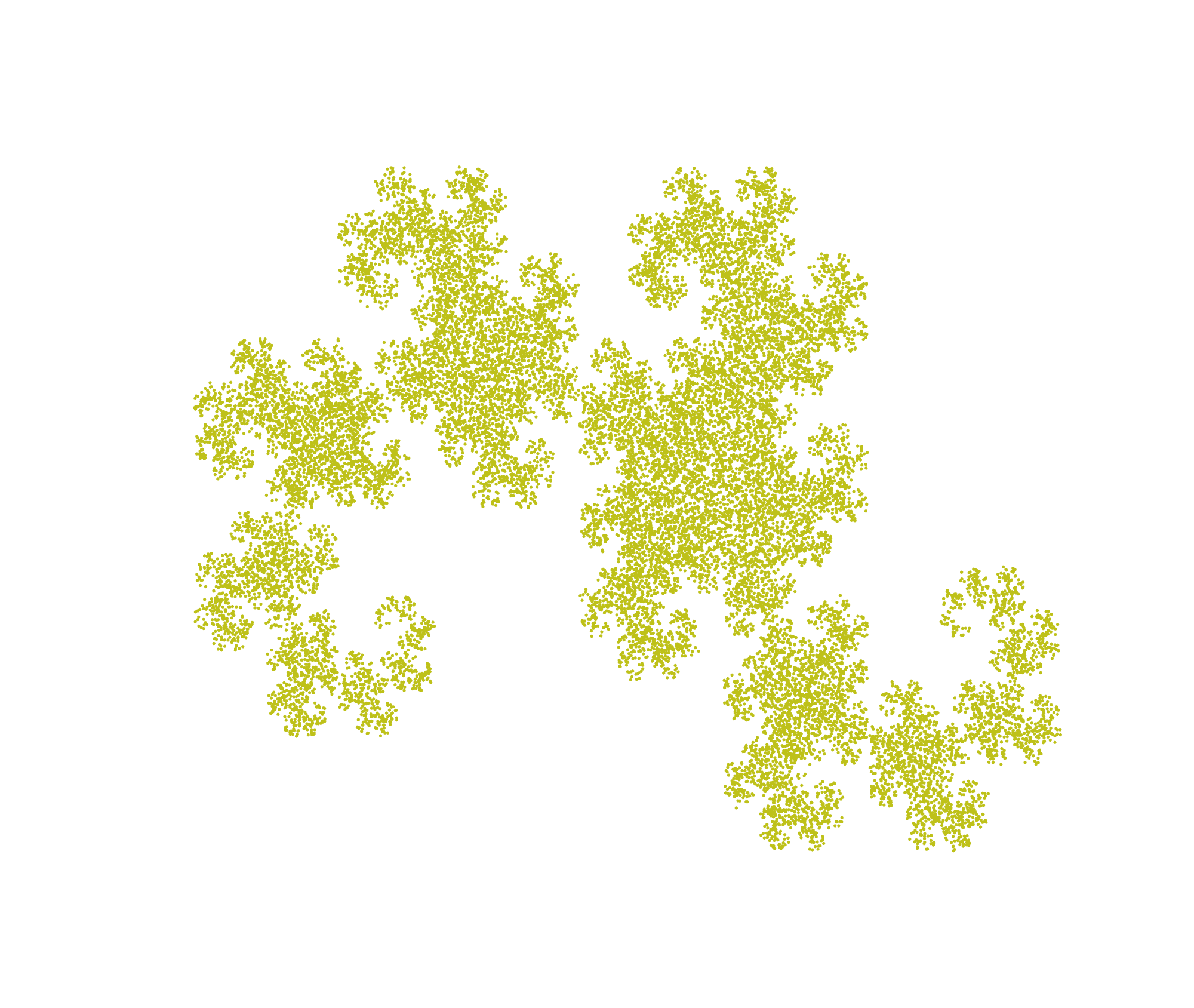}
\qquad \qquad \qquad
\includegraphics[height=4cm,width=4cm, bb=0 0 1000 1000]{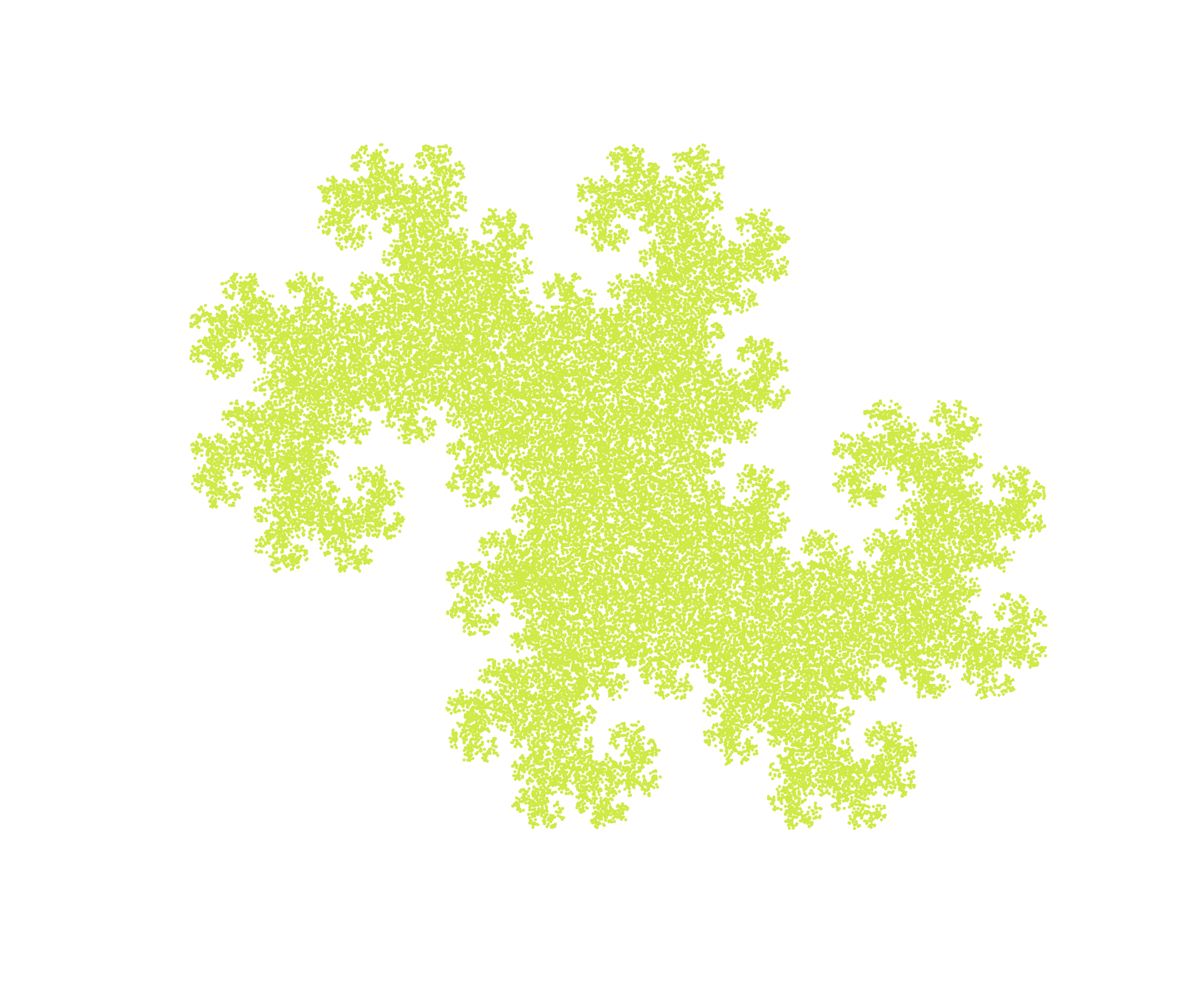}
\caption{The Heighway dragon and the Twindragon}
\label{figure1}
\end{figure}

The Heighway dragon and the Twindragon are both self-similar sets, not generated by \eqref{eq:IFS1}, but by a slightly different pair of two similar contractions:
\begin{equation}
\label{eq:IFS2} 
\begin{cases}
 \phi_1(z)=\alpha z,& \\
 \phi_2(z)=(\alpha e^{i \theta}) z+1.
\end{cases}
\end{equation}
In particular, if $\alpha=(1+i)/2$ and $\theta=\pi/2$, then \eqref{eq:IFS2} generates the Heighway dragon. If $\alpha=(1+i)/2$ and $\theta=\pi$, then \eqref{eq:IFS2} generates the Twindragon~\cite{Davis+Knuth-1970}. 
\medskip

The following questions arise naturally. How can we give a parametrized expression for self-similar sets generated by the IFS \eqref{eq:IFS2}? Does this small change from $\alpha$ to $1$ give any influence to the relationship with generalized revolving sequences? 

\medskip

Let $H_{1,\alpha,\theta}$ be the self-similar set generated by the IFS \eqref{eq:IFS2}. Notice that $\psi_1(z)=\alpha \phi_1(z/\alpha)$ and $\psi_2(z)=\alpha \phi_2(z/\alpha)$ in \eqref{eq:IFS1} and \eqref{eq:IFS2}. Thus, 
\begin{equation*}
X_{1,\alpha,\theta}=\psi_1(X_{1,\alpha,\theta}) \cup \psi_2(X_{1,\alpha,\theta})
                    =\alpha \phi_1 \left(\frac{X_{1,\alpha,\theta}}{\alpha}\right) \cup \alpha \phi_2 \left(\frac{X_{1,\alpha,\theta}}{\alpha}\right).
\end{equation*}
Therefore, 
\begin{equation*}
\frac{X_{1,\alpha,\theta}}{\alpha}=\phi_1 \left(\frac{X_{1,\alpha,\theta}}{\alpha}\right) \cup \phi_2 \left(\frac{X_{1,\alpha,\theta}}{\alpha}\right).
\end{equation*}

Since $H_{1,\alpha, \theta}$ is a unique non-empty compact solution of the IFS \eqref{eq:IFS2} and $X_{1,\alpha,\theta}$ is a closed set, we have 
$$H_{1,\alpha, \theta}=X_{1,\alpha,\theta}/\alpha.$$ 

Therefore, the following proposition follows immediately from \eqref{eq:X1}.

\begin{proposition}
\label{prop:toby1} 
Let $H_{1,\alpha,\theta}$ be the self-similar set generated by the IFS: 
\begin{equation*}
\begin{cases}
 \phi_1(z)=\alpha z,& \\
 \phi_2(z)=(\alpha e^{i \theta}) z+1.
\end{cases}
\end{equation*}

Then $H_{1,\alpha,\theta}$ has the following parametrized expression. 
\begin{equation}
  H_{1,\alpha,\theta}=\left\{ \sum_{n=0}^{\infty} \delta_{n}\alpha^{n}: \delta_{j_1}=1, (\delta_{0},\delta_{1},\dots )\in W_{\theta} \right\}
	=X_{1,\alpha,\theta}/\alpha,
\end{equation}
where $j_{1}:=\min\{j: \delta_{j} \not= 0\}$. 
\end{proposition}

\begin{remark}
Define
\begin{align*}
H_{\alpha,\theta}:&=\left\{ \sum_{n=0}^{\infty} \delta_{n}\alpha^{n}: (\delta_{0}, \delta_{1},\dots )\in W_{\theta} \right\}, \\ 
X_{\alpha,\theta}:&=\left\{ \sum_{n=1}^{\infty} \delta_{n}\alpha^{n}: (\delta_{1}, \delta_{2},\dots )\in W_{\theta} \right\}.
\end{align*}
Notice that 
\begin{equation*}
H_{\alpha,\theta}=\bigcup_{l=0}^{p-1}(e^{i\theta})^{l}H_{1,\alpha,\theta}, \qquad 
X_{\alpha,\theta}=\bigcup_{l=0}^{p-1}(e^{i\theta})^{l}X_{1,\alpha,\theta}.
\end{equation*}

Figure \ref{figure2} also confirms that two slightly different pairs IFSs can generate a geometrically similar attractor. It is interesting to note that the only difference in the parametric expressions of $H_{\alpha,\theta}$ and $X_{\alpha,\theta}$ is whether the sum begins from $n=0$ or $n=1$. 
\end{remark}

\begin{figure}
\includegraphics[height=4.5cm,width=4.5cm, bb=0 0 1000 1000]{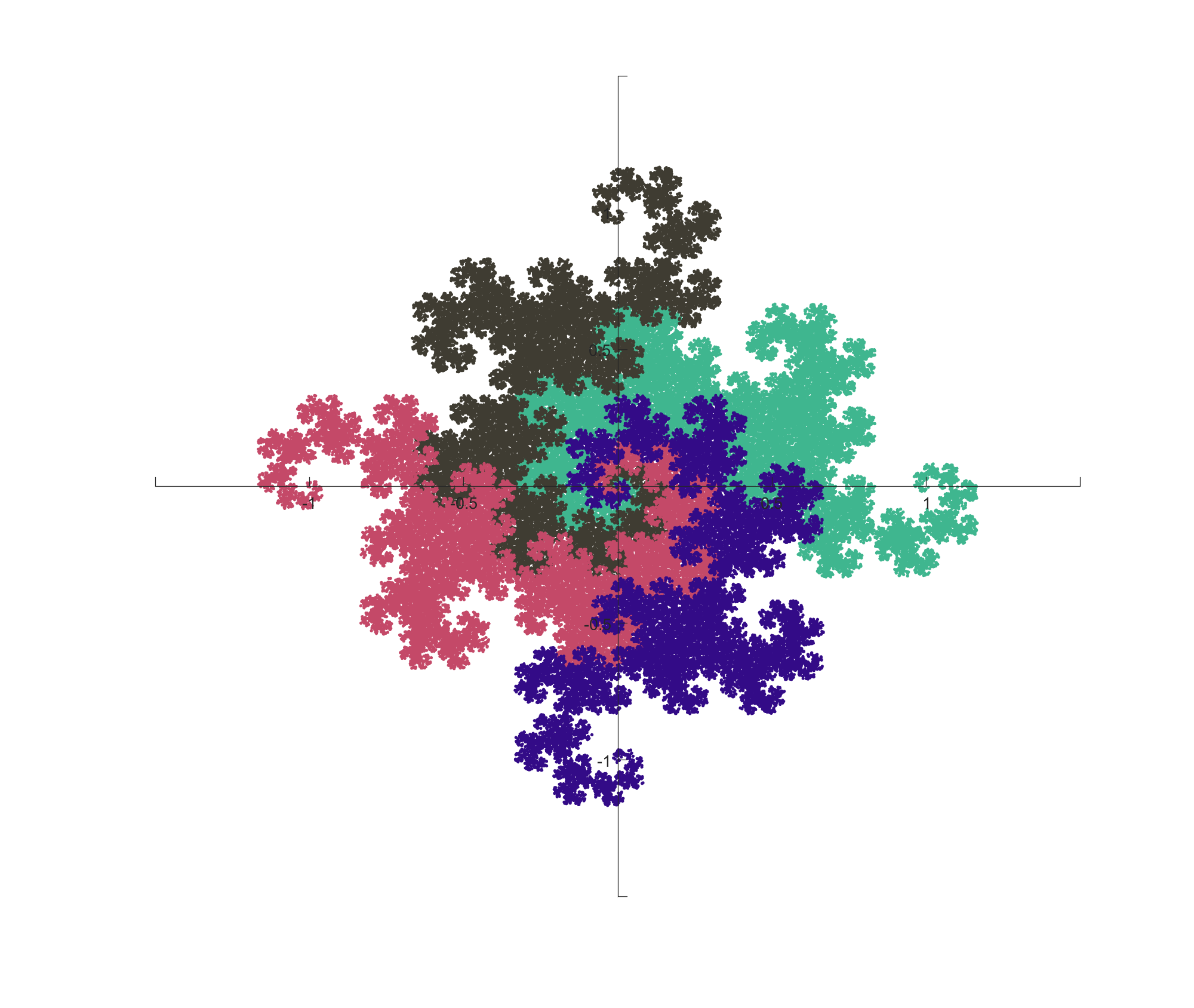}
\qquad \qquad \qquad 
\includegraphics[height=4.5cm,width=4.5cm, bb=0 0 1000 1000]{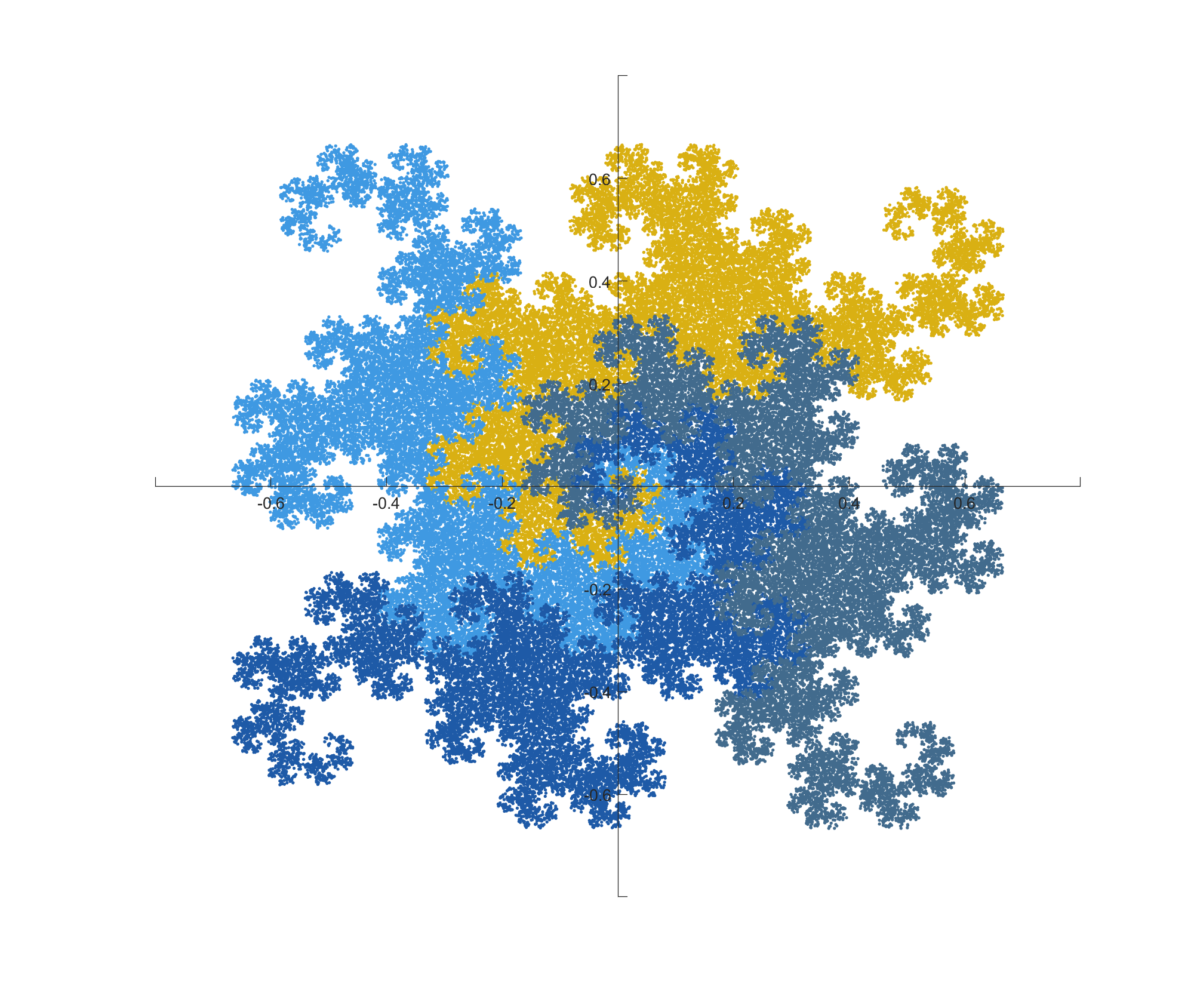} 
\caption{$H_{\alpha,\theta}$ (left) and $X_{\alpha,\theta}$ (right): $(\alpha, \theta)=((1+i)/2, \pi/2)$}
\label{figure2}
\end{figure}

\section{Structure of attractor and signed revolving sequences}

Proposition~\ref{prop:toby1} suggests a natural question: Does a small change in the constant term of the second map from $\alpha$ to $1$ always generate a geometrically similar attractor? Is it always true that this small change does not give any major influence to the parametrized expression for certain self-similar sets? 

\medskip

To begin, we recall another result from~\cite{Kawamura-Allen-2020}. 

\begin{definition} \label{def:SRC}
A sequence ${(\delta_{1},\delta_{2},\dots)\in\Delta_{\theta}^{\NN}}$ satisfies the {\em Signed Revolving Condition (SRC)} if  
\begin{enumerate}
\item $\delta_1$ is free to choose, 
\item If $\delta_{1}=\delta_{2}=\dots =\delta_{k}=0$, then  $\delta_{k+1}$ is free to choose,
\item Otherwise, $\delta_{k+1}=0$ or \begin{align*}
\delta_{k+1}=\begin{cases} 	
(e^{+i\theta})\delta_{j_{0}(k)}, &\mbox{if $j_0(k)$ is odd},\\ 
(e^{-i\theta})\delta_{j_{0}(k)}, &\mbox{if $j_0(k)$ is even}, 
\end{cases}
\end{align*}
where $j_{0}(k):=\max \{j \le k : \delta_{j} \not= 0 \}$. 
\end{enumerate} 
\end{definition}

Compared to the generalized revolving sequences, which always move in the same direction, we see that the direction of movement of the sequence $(\delta_n)$ depends on its past. Define $W^{\pm}_{\theta}$ as the set of all signed revolving sequences with parameter $\theta$. 

\begin{theorem}[Kawamura-Allen]
\label{th:main2}
Let $X^{2}_{1,\alpha,\theta}$ be the self-similar set generated by the iterated function system: 
\begin{equation}
\label{eq:second case}
 \begin{cases}
 \psi_1(z)=\alpha \overline{z}, &\\
 \psi_2(z)=(\alpha e^{i \theta}) \overline{z}+\alpha, 
 \end{cases}
\end{equation}
where $\alpha \in \CC$ such that $|\alpha|<1$.

Then $X_{1,\alpha,\theta}^{2}$ has the following parametrized expression:
\begin{equation}
\label{eq:X2}
  X_{1,\alpha,\theta}^{2}=\left\{ \sum_{n=1}^{\infty} \delta_{n} \prod_{j=1}^{n}\eta_j 
: \delta_{j_1}=1, (\delta_{1}, \delta_{2},\dots )\in W^{\pm}_{\theta} \right\}, 
\end{equation}
where $j_{1}:=\min\{j: \delta_{j} \not= 0\}$ and 
\begin{equation*}
\eta_j=
\begin{cases}
\alpha, & \mbox{if $j$ is odd}, \\
\overline{\alpha}, & \mbox{if $j$ is even}.
\end{cases}
\end{equation*}
\end{theorem}

\begin{remark}
Define
\begin{align*}
H_{\alpha,\theta}^{2}:&=\left\{ \sum_{n=0}^{\infty} \delta_{n} \prod_{j=1}^{n}\eta_j 
: (\delta_{0}, \delta_{1},\dots )\in W^{\pm}_{\theta} \right\}, \\
X_{\alpha,\theta}^{2}:&=\left\{ \sum_{n=1}^{\infty} \delta_{n} \prod_{j=1}^{n}\eta_j 
: (\delta_{1}, \delta_{2},\dots )\in W^{\pm}_{\theta} \right\}.
\end{align*}

Surprisingly, computer simulations suggest that changing the starting point from $n=1$ to $n=0$ generates a geometrically very different attractor; see Figure~\ref{figure3}.
\end{remark}

\begin{figure}
\includegraphics[height=4.5cm,width=4.5cm, bb=0 0 1000 1000]{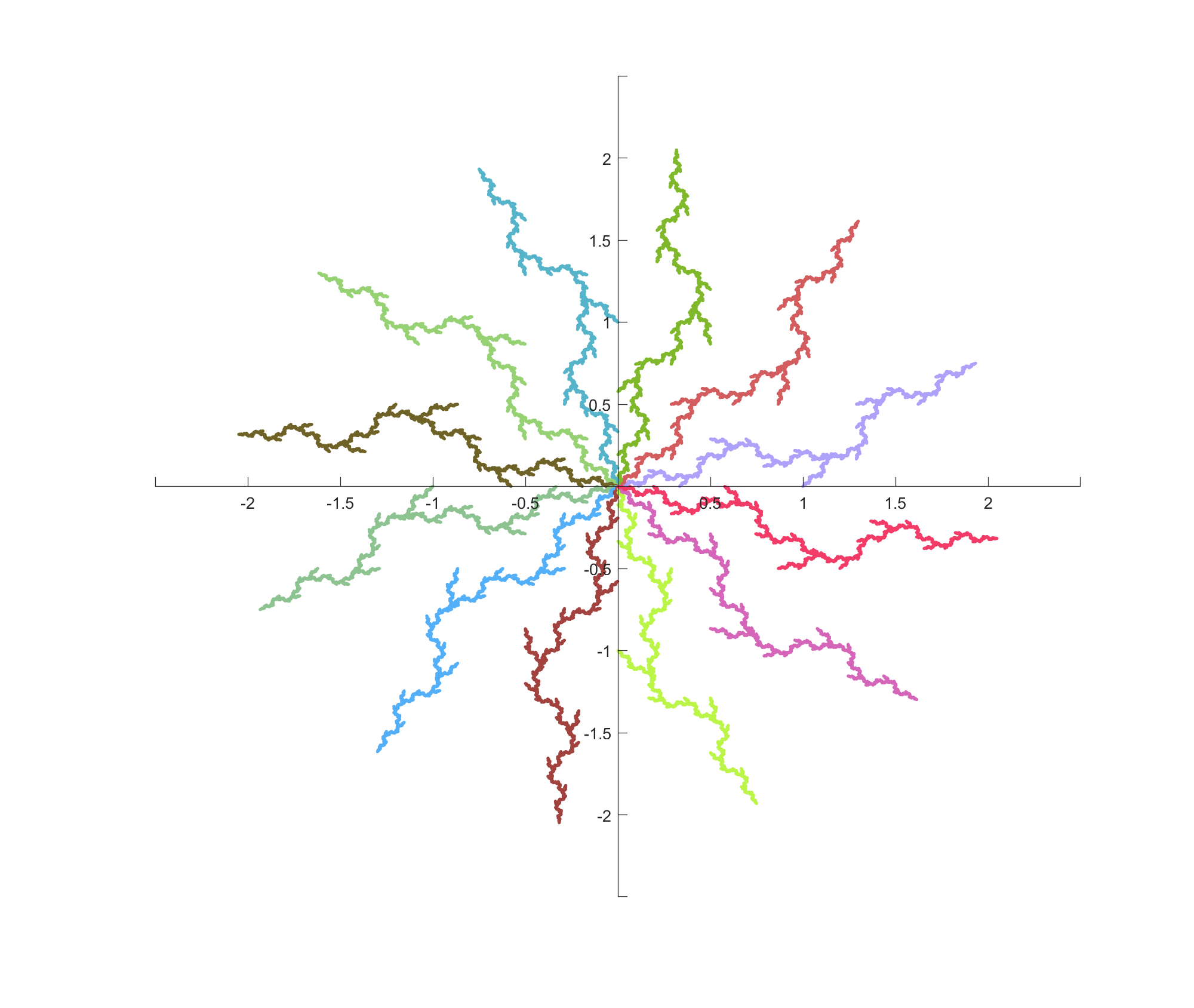}
\qquad \qquad \qquad 
\includegraphics[height=4.5cm,width=4.5cm, bb=0 0 1000 1000]{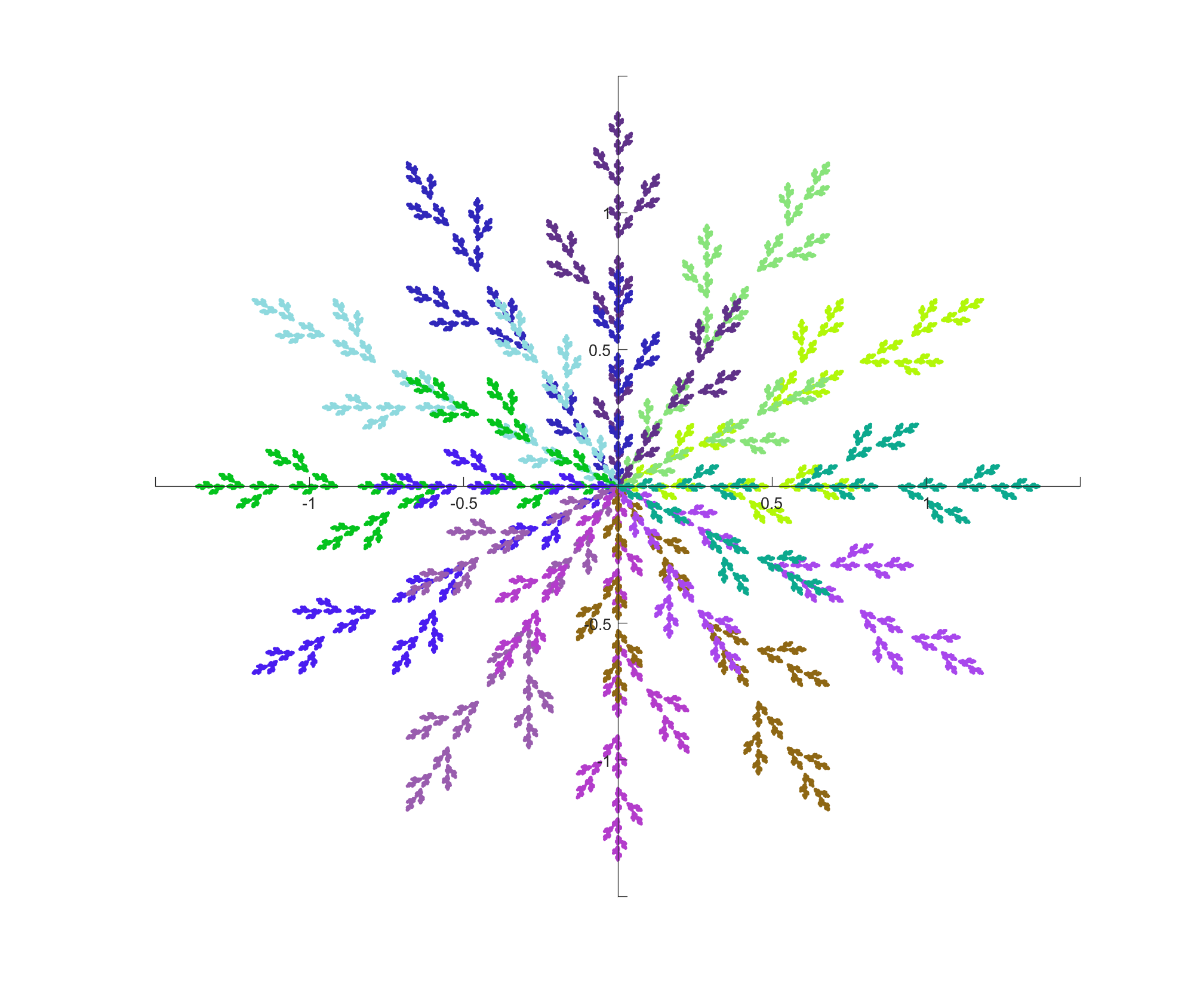}
\caption{$H^{2}_{\alpha,\theta}$ (left) and $X^{2}_{\alpha,\theta}$ (right): $(\alpha, \theta)=(\frac12+\frac{\sqrt{3}i}{6}, \frac{\pi}{6})$}
\label{figure3}
\end{figure}

Let $X_{1,\overline{\alpha},\theta}^{2}$ be the self-similar set generated by
\begin{equation}
\label{alpha-bar}
 \begin{cases}
 \psi_1(z)=\overline{\alpha} \cdot \overline{z},& \\
 \psi_2(z)=(\overline{\alpha} e^{i\theta}) \overline{z} + \overline{\alpha}.
 \end{cases}
\end{equation}

Notice that Theorem~\ref{th:main2} applies to all $\alpha \in \CC$ such that $|\alpha|<1$, so that $X_{1,\overline{\alpha},\theta}^{2}$ 
has the following parametrized expression:
\begin{equation*}
        X_{1,\overline{\alpha},\theta}^2=\left\{ \sum_{n=1}^{\infty} \delta_{n} \prod_{j=1}^{n} \overline{\eta_j}: \delta_{j_1}=1, (\delta_{1}, \delta_{2},\dots )\in W^{\pm}_{\theta} \right\}, 
\end{equation*}
where $j_{1}:=\min\{j: \delta_{j} \not= 0\}$ and $\eta_1=\alpha$ and $\eta_{j+1}=\overline{\eta_j}$ for $j=1,2, \dots$. This leads to the  following proposition:
\begin{proposition}
        Let $H_{1,\alpha,\theta}^2$ be the self-similar set generated by
            \begin{equation}
						\label{eq:H2}
            \displaystyle \begin{cases}
                \phi_1(z)=\alpha \overline{z} & \\
                \phi_2(z)=(\alpha e^{i\theta}) \overline{z} + 1.
            \end{cases}
        \end{equation}
        
        Then $H_{1,\alpha,\theta}^2$ has the following parametrized expression
        
        \begin{equation*}
            H_{1,\alpha,\theta}^2=\left\{ \sum_{n=0}^{\infty} \delta_{n} \prod_{j=1}^{n} \eta_j: \delta_{j_1}=1, (\delta_{0}, \delta_{1},\dots )\in W^{\pm}_{\theta} \right\}=X_{1,\overline{\alpha},\theta}^2 / \overline{\alpha}, 
\end{equation*}
where $j_{1}:=\min\{j: \delta_{j} \not= 0\}$ and $\eta_1=\alpha$ and $\eta_{j+1}=\overline{\eta_j}$ for $j=1,2, \dots$.
\end{proposition}

\begin{proof}
From \eqref{alpha-bar} and \eqref{eq:H2}, notice that $\psi_1(z)=\overline{\alpha}\phi_1(z / \overline{\alpha})$ and $\psi_2(z)=\overline{\alpha}\phi_2(z / \overline{\alpha})$.  Thus, we have 
    \begin{equation*}
        X_{1,\overline{\alpha},\theta}^2 = \psi_1(X_{1, \overline{\alpha},\theta}^2)\cup\psi_2(X_{1,\overline{\alpha},\theta}^2) 
        = \overline{\alpha}\phi_1(X_{1,\overline{\alpha},\theta}^2 / \overline{\alpha})\cup\overline{\alpha}\phi_2(X_{1,\overline{\alpha},\theta}^2/\overline{\alpha}).
    \end{equation*}
    
    Since $H_{1,\alpha,\theta}^2$ is the unique non-empty compact solution of \eqref{eq:H2} and $X_{1,\overline{\alpha},\theta}^2$ is a closed set,
    \begin{align*}
        H_{1,\alpha,\theta}^2=X_{1,\overline{\alpha},\theta}^2 / \overline{\alpha}
				&=\left\{ \sum_{n=1}^{\infty} \delta_{n} \prod_{j=2}^{n} \overline{\eta_j}: \delta_{j_1}=1, (\delta_{1}, \delta_{2},\dots )\in W^{\pm}_{\theta} \right\} \\
				&=\left\{ \sum_{n=1}^{\infty} \delta_{n} \prod_{j=1}^{n-1} \eta_j: \delta_{j_1}=1, (\delta_{1}, \delta_{2},\dots )\in W^{\pm}_{\theta} \right\}.
    \end{align*}
    \end{proof}

\begin{remark}

It is interesting to notice that changing the constant term in the IFS \eqref{eq:second case} from $\alpha$ to $1$ generates a very different attractor since $H_{1,\alpha,\theta}^2 \neq X_{1,\alpha,\theta}^2 / \alpha$ but rather $H_{1,\alpha,\theta}^2=X_{1,\overline{\alpha},\theta}^2 / \overline{\alpha}$. However, the only difference in the parametric expressions of $H_{1,\alpha,\theta}^2$ and $X_{1,\alpha,\theta}^2$ is whether the sum begins from $n=0$ or $n=1$. 
\end{remark}

\section{Terdragon and Ternary Revolving Sequences} 

L\'evy dragon curve, the Heighway dragon and the Twindragon are all generated by two similar contractions; however, the Terdragon is generated by three similar contractions:
\begin{equation}
\label{eq:IFS3}
\begin{cases}
 \phi_1(z)=\alpha z,& \\
 \phi_2(z)=(\alpha e^{i \theta}) z+\alpha, & \\
 \phi_3(z)=\alpha z + \beta.
\end{cases}
\end{equation}
In particular, if $\alpha=\frac12-\frac{\sqrt{3}i}{6}, \beta=\bar{\alpha}$ and $\theta=\frac{2\pi}{3}$, then the IFS \eqref{eq:IFS3} generates the Terdragon.
\medskip 

A question arises naturally: How can we give a parametrized expression for self-similar sets generated by the IFS \eqref{eq:IFS3}? How does the number of contractions influence the properties of generalized revolving sequences?

\begin{definition}
A sequence ${(\delta_{n})\in\Delta_{\theta}^{\NN}}$ satisfies the {\em Ternary Revolving Condition (TRC)} if 
        \begin{enumerate}
            \item $\delta_{1}$ is free to choose,
            \item If $\delta_{1}=\delta_{2}=\dots=\delta_{k}=0$ then $\delta_{k+1}$ is free to choose,
            \item Otherwise, $\delta_{k+1}=0$, $\delta_{k+1}=\delta_{j_{0}(k)}$ or $\delta_{k+1}=(e^{i\theta})\delta_{j_{0}(k)}$
				\end{enumerate}
where $j_{0}(k)=\max\{j\leq k:\delta_{j}\neq 0\}$.        
\end{definition}

Loosely speaking we can say that a ternary revolving sequence $(\delta_{n})\in\Delta_{\theta}^{\NN}$ is a sequence with nonzero terms lying on the unit circle and with subsequent terms which either:
	\begin{enumerate}
			\item Stay in place,
			\item Move forward on the unit circle by angle $\theta$,
			\item Or go to the origin.
	\end{enumerate}

Define $W_{\theta}^{3}$ to be the set of all ternary revolving sequences with parameter $\theta$:
$$W_{\theta}^{3}:=\{(\delta_{n})\in\Delta_{\theta}^{\mathbb{N}} :(\delta_{n}) \text{ satisfies the TRC}\}.$$

\begin{definition} 
    Let $\Sigma:=(\delta_{n})\in W_{\theta}^{3}$. Define the function $b_{\Sigma}:\mathbb{N}\longrightarrow\{0,1\}$ as follows: 
	\begin{equation*}
		b_{\Sigma}(n):=
			\begin{cases}
				1, \text{ if }\delta_{n}=\delta_{j_{1}(n)} \\
				0, \text{ otherwise}
			\end{cases}
	\end{equation*}
	where $j_{1}(n):=\min\{j>n:\delta_{j}\neq 0\}$. Fix $b_{\Sigma}(n)=0$, if $\delta_{j}=0$ $\forall j>n$.
	\medskip
	
	The Binary Static Sequence (BSS) for $(\delta_{n})$ is the binary sequence $(b_{n})$ such that $b_{n}=b_{\Sigma}(n)$, $\forall n \in\mathbb{N}$.
\end{definition}

For example, consider the following segment of a ternary revolving sequence $(\delta_{n})$ with parameter $\theta=\frac{\pi}{2}$. 
$$\boldsymbol{(\delta_{n}):}\: 1 \longrightarrow 0 \longrightarrow 1 \longrightarrow i \longrightarrow i \longrightarrow -1 \dots$$
Then the corresponding binary static sequence $(b_{n})$ is     
$$\boldsymbol{(b_{n}):}\: 1 \longrightarrow 0 \longrightarrow 0 \longrightarrow 1 \longrightarrow 0 \longrightarrow \text{?} \dots$$

Note that because the BSS for $(\delta_{n})$ is \textit{future dependent} we cannot know what the last shown element in $(b_{n})$ is without knowing what comes next in $(\delta_{n})$. 

Loosely speaking, $(b_{n})$ tells you when the the sequence $(\delta_{n})$ stays in place on the unit circle. 
Also note that $b_{\Sigma}(n)=0$ for all $n \in \NN$ such that $\delta_{n}=0$.

\medskip
For given $\alpha,\beta\in\mathbb{C}$ such that $|\alpha|<1$ , $|\beta|<1$ and $|\theta|=\frac{2 \pi q}{p}\leq\pi$, define $T_{\alpha,\beta,\theta}$ as follows:    
    \begin{equation*}
        T_{\alpha,\beta,\theta}:=\left\{\sum_{n=1}^{\infty}\delta_{n}\alpha^{n} (\beta / \alpha)^{b_n}: (\delta_n) \in W_{\theta}^{3}, (b_n) \text{ is the BSS for } (\delta_n) \right\}.
    \end{equation*} 

\begin{figure}
\includegraphics[height=4.5cm,width=4.5cm, bb=0 0 800 800]{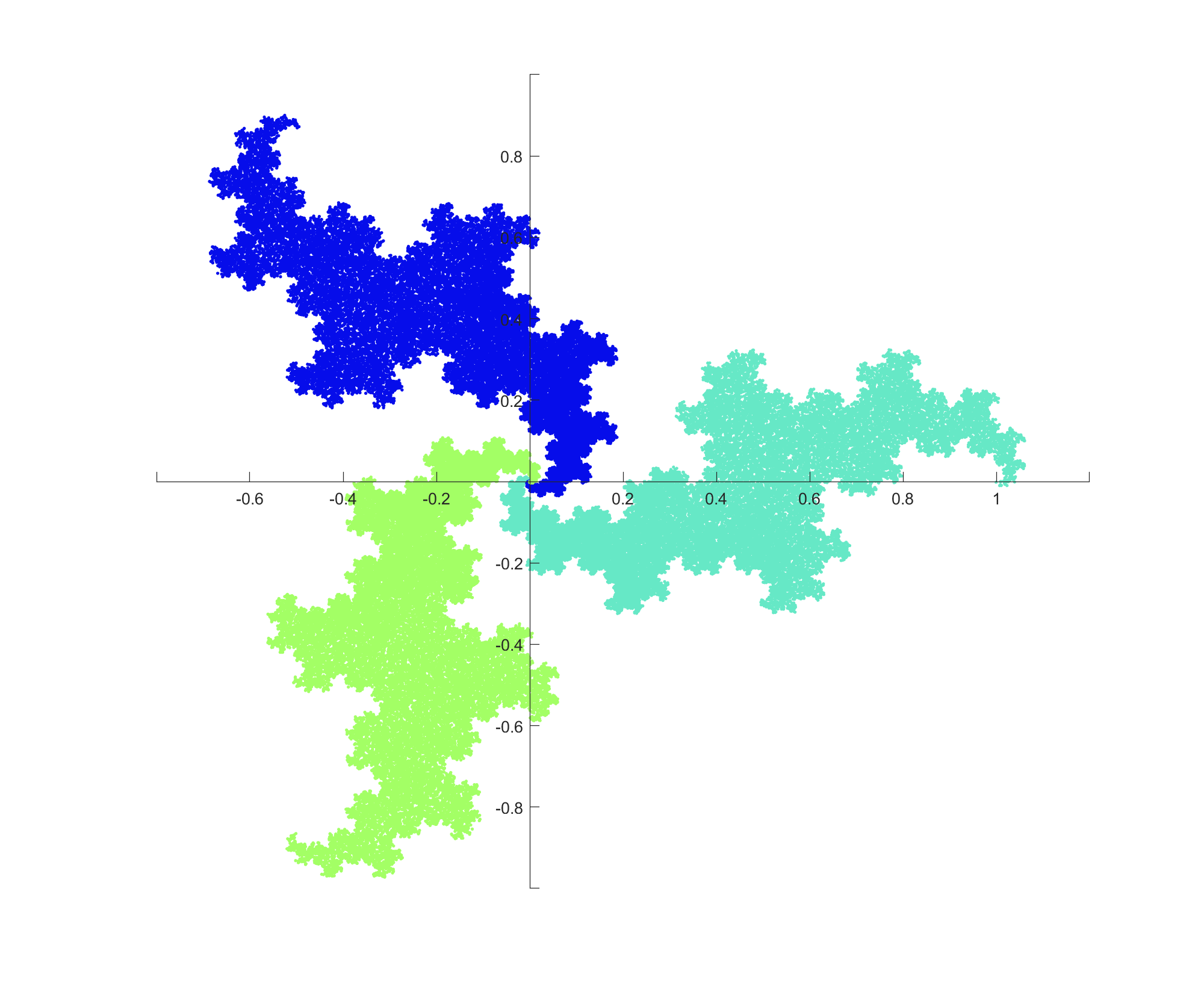}
\qquad \qquad \qquad 
\includegraphics[height=4.5cm,width=4.5cm, bb=0 0 800 800]{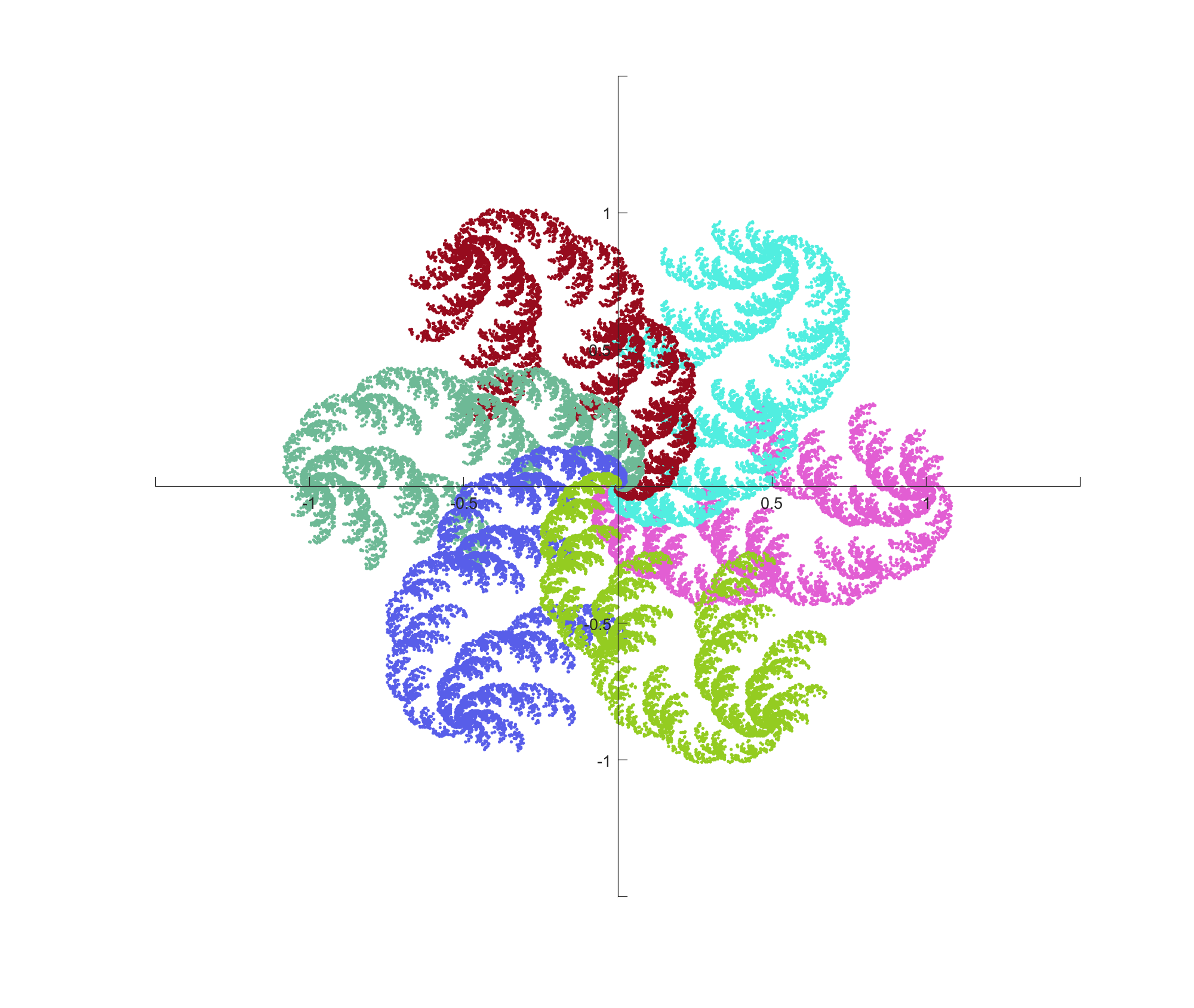}
\caption{$T_{\alpha, \beta, \theta}:\alpha=\frac12-\frac{\sqrt 3i}{6}, \beta=\bar{\alpha}$, $\theta=\frac{2\pi}{3}$ (left) and $\theta=\frac{\pi}{3} $ (right).} 
\label{figure4} 
\end{figure}

See Figure~\ref{figure4}. Computer simulations suggest that $T_{\alpha, \beta, \theta}$ must be a union of three Terdragons as a special case. 
Finally, we will prove the following main theorem.  

\medskip
Define a subset of $T_{\alpha,\beta,\theta}$ as follows.
\begin{equation*}
        T_{1,\alpha,\beta,\theta}:=\left\{\sum_{n=1}^{\infty}\delta_{n}\alpha^{n} (\beta / \alpha)^{b_n}:\delta_{j_{0}} = 1, (\delta_n) \in W_{\theta}^{3}, (b_n) \text{ is the BSS for } (\delta_n) \right\}
    \end{equation*}    
    where $j_{0}=\min\{j:\delta_{j}\neq 0\}$. 

\begin{theorem}
		\label{theorem:set equation}
		$T_{1,\alpha,\beta,\theta}$ satisfies the set equation
    \begin{equation*}
        T_{1,\alpha,\beta,\theta}=\Psi_1(T_{1,\alpha,\beta,\theta})\cup\Psi_2(T_{1,\alpha,\beta,\theta})\cup\Psi_3(T_{1,\alpha,\beta,\theta}),
    \end{equation*}
    where 
		\begin{equation}
\label{eq: IFS-T}
        \begin{cases}
            \Psi_1(z)=\alpha z, \\
            \Psi_2(z)=(\alpha e^{i\theta})z+\alpha, \\
            \Psi_3(z)=\alpha z + \beta,
        \end{cases}
    \end {equation}
where $\alpha,\beta\in\mathbb{C}$ such that $|\alpha|<1$ , $|\beta|<1$ and $|\theta|=\frac{2 \pi q}{p}\leq\pi$. 
\end{theorem}

\begin{proof} 

Let $z\in T_{1,\alpha,\beta,\theta}$. Then $z=\sum_{n=1}^{\infty}\delta_{n}\alpha^{n}(\beta/\alpha)^{b_n}$ for some $(\delta_{n})\in W_{\theta}^{3}$ and the corresponding BSS $(b_{n})$. Consider the ordered pair of points $(\delta_n, b_n)$. Since the first non-zero digit of $\delta_n$ is $1$, the only possible values for $(\delta_1,b_1)$ are $(0,0),(1,0)$, and $(1,1)$. 

\medskip

If $(\delta_{1},b_{1})=(0,0)$, set $\delta_{n}^{'}:=\delta_{n+1}$ and $b_{n}^{'}:=b_{n+1}$ for $n=1, 2, \dots$. Then $(\delta_{n}^{'})$ satisfies the TRC with its first nonzero digit equal to $1$ and $(b_{n}^{'})$ is the BSS for $(\delta_{n}^{'})$. Thus, we have  
	\begin{equation*}
	    z = \sum_{n=1}^{\infty}\delta_{n+1}\alpha^{n+1}(\beta /\alpha)^{b_{n+1}} 
			=\alpha\sum_{n=1}^{\infty}\delta_{n}^{'}\alpha^{n}(\beta /\alpha)^{b_{n}^{'}}\in\Psi_{1}(T_{1,\alpha,\beta,\theta}).
	\end{equation*}    

\medskip

If $(\delta_{1},b_{1})=(1,0)$, set $\delta_{n}^{'}:=(e^{-i\theta})\delta_{n+1}$ and $b_{n}^{'}:=b_{n+1}$ for $n=1,2,\dots$. Since $b_{1}=0$, the next nonzero digit of $\delta_{n}$ is $e^{i\theta}$. Then the sequence $(\delta_{n}^{'})$ satisfies the TRC with its first nonzero digit equal to $1$ and $(b_{n}^{'})$ is the BSS for $(\delta_{n}^{'})$. Thus, we have 
\begin{align*}
z &= \alpha + \sum_{n=1}^{\infty}\delta_{n+1}\alpha^{n+1}(\beta /\alpha)^{b_{n+1}}
   = \alpha+\alpha(e^{i\theta})\sum_{n=1}^{\infty}(e^{-i\theta})\delta_{n+1}\alpha^{n}(\beta /\alpha)^{b_{n+1}} \\
	&=\alpha+\alpha(e^{i\theta})\sum_{n=1}^{\infty}\delta_{n}^{'}\alpha^{n}(\beta /\alpha)^{b_{n}^{'}} \in\Psi_{2}(T_{1,\alpha,\beta,\theta}).
	\end{align*}

\medskip
	
If $(\delta_{1},b_{1})=(1,1)$, set $\delta_{n}^{'}:=\delta_{n+1}$ and $b_{n}^{'}:=b_{n+1}$. Since $b_{1}=1$, the next nonzero digit of $\delta_{n}$ is $1$. Then the sequence $(\delta_{n}^{'})$ satisfies the TRC with its first nonzero digit equal $1$  and $(b_{n}^{'})$ is the BSS for $(\delta_{n}^{'})$. Thus, we have 

	\begin{equation*}
	    z = \beta + \sum_{n=1}^{\infty}\delta_{n+1}\alpha^{n+1}(\beta /\alpha)^{b_{n+1}} 
	    = \beta + \alpha\sum_{n=1}^{\infty}\delta_{n}^{'}\alpha^{n}(\beta /\alpha)^{b_{n}^{'}} \in \Psi_{3}(T_{1,\alpha,\beta,\theta}).
	\end{equation*}
Thus, 
$$T_{1,\alpha,\beta,\theta} \subset \Psi_{1}(T_{1,\alpha,\beta,\theta})\cup \Psi_{2}(T_{1,\alpha,\beta,\theta}) \cup \Psi_{3}(T_{1,\alpha,\beta,\theta}).$$

\bigskip
The reverse inclusion follows analogously. If $z \in \Psi_{1}(T_{1, \alpha,\beta,\theta})$, then 
\begin{equation*}
z=\alpha \sum_{n=1}^{\infty}\delta_n\alpha^{n}(\beta/\alpha)^{b_n} 
\end{equation*}
for some $(\delta_{n})\in W_{\theta}^{3}$ and the corresponding BSS $(b_{n})$. Set 
\begin{equation*}
		\delta_j^{'}:=
			\begin{cases}
				0, &\text{ if }j=1 \\
				\delta_{j-1}, &\text{ if } j \geq 2, 
			\end{cases}
			\qquad 
		b_j^{'}:=
			\begin{cases}
				0, &\text{ if }j=1 \\
				b_{j-1}, &\text{ if } j \geq 2.
			\end{cases}
	\end{equation*}
	Since $(\delta_j^{'})$ satisfies the TRC with first nonzero digit equal to $1$ and $(b_j^{'})$ is the BSS for$(\delta_j^{'})$, we have   
	\begin{equation*}
    z=\sum_{n=2}^{\infty}\delta_{n-1}\alpha^{n}(\beta / \alpha)^{b_{n-1}}
		=\sum_{j=1}^{\infty}\delta_{j}^{'}\alpha^{j}(\beta / \alpha)^{b_{j}^{'}} \in T_{1,\alpha,\beta,\theta}.
	\end{equation*}

\medskip
  
If $z \in \Psi_{2}(T_{1,\alpha,\beta,\theta})$, then 
\begin{equation*}
z=\alpha+(\alpha e^{i\theta})\sum_{n=1}^{\infty}\delta_n\alpha^{n}(\beta/\alpha)^{b_n}.
\end{equation*}
for some $(\delta_{n})\in W_{\theta}^{3}$ and the corresponding BSS $(b_{n})$. Set   
\begin{equation*}
		\delta_j^{'}:=
			\begin{cases}
				1, &\text{ if }j=1 \\
				e^{i \theta}\delta_{j-1}, &\text{ if } j \geq 2 
			\end{cases}
	\qquad 
		b_j^{'}:=
			\begin{cases}
				0, &\text{ if }j=1 \\
				b_{j-1}, &\text{ if } j \geq 2.
			\end{cases}
	\end{equation*}
	
Notice that the sequence $(\delta_j^{'})$ satisfies the TRC with the first nonzero digit equal to $1$. Thus,  
	\begin{equation*}
    z=\delta_1^{'} \alpha (\beta/ \alpha)^{0}+\sum_{j=2}^{\infty}(e^{i \theta}\delta_{j-1}) \alpha^{j}(\beta / \alpha)^{b_{j-1}}
		=\sum_{j=1}^{\infty}\delta_{j}^{'}\alpha^{j}(\beta /\alpha)^{b_{j}^{'}} \in T_{1,\alpha,\beta,\theta}.
	\end{equation*}

Finally, if $z \in \Psi_{3}(T_{1,\alpha,\beta,\theta})$, then 
\begin{equation*}
z=\beta+\alpha \sum_{n=1}^{\infty}\delta_{n}\alpha^{n}(\beta / \alpha)^{b_n}.
\end{equation*}
Set 
	\begin{equation*}
		\delta_j^{'}:=
			\begin{cases}
				1, &\text{ if }j=1 \\
				\delta_{j-1}, &\text{ if } j \geq 2 
			\end{cases}
	\qquad 
	b_j^{'}:=
			\begin{cases}
				1, &\text{ if }j=1 \\
				b_{j-1}, &\text{ if} j \geq 2.
			\end{cases}
	\end{equation*} 
	Notice that both the first and second nonzero digits of $(\delta_j^{'})$ are $1$ and $(\delta_j^{'})$ satisfies the TRC.  
	Since $b_1^{'}=1$, $(b_{j}^{'})$ is the corresponding BSS for $(\delta_j^{'})$. So, 
	\begin{equation*}
    z=\beta+\sum_{n=2}^{\infty}\delta_{n-1}\alpha^{n}(\beta / \alpha)^{b_{n-1}}
		=\sum_{j=1}^{\infty}\delta_{j}^{'}\alpha^{j}(\beta / \alpha)^{b_{j}^{'}} \in T_{1,\alpha,\beta,\theta}.
	\end{equation*}
As a result, 
$$\Psi_{1}(T_{1,\alpha,\beta,\theta})\cup\Psi_{2}(T_{1,\alpha,\beta,\theta})\cup\Psi_{3}(T_{1,\alpha,\beta,\theta}) \subset T_{1,\alpha,\beta,\theta}.$$
\end{proof}

\begin{corollary}
The closure of $T_{1,\alpha,\beta,\theta}$ is the self-similar set generated by the iterated function system \eqref{eq: IFS-T}.
\end{corollary}

\begin{remark}
We do not know if $T_{1,\alpha,\beta,\theta}$ is closed. The technique used in \cite{Kawamura-Allen-2020} to prove that $X_{1, \alpha, \theta}$ is closed can not be used for $T_{1,\alpha,\beta,\theta}$. The issue is that the BSS $(b_n)$ for $(\delta_n)$ is future dependent.  
\end{remark}

\section*{Acknowledgment}
The second author greatly appreciates Prof.~P.~Allaart for his helpful comments and suggestions in preparing this paper.

\end{document}